\documentclass{dmgtuu}

\usepackage{amssymb, amsmath}
\usepackage[IL2]{fontenc}
\usepackage{graphicx, color}

\newauthor{J\'ulius Czap}{J. Czap}
{Department of Applied Mathematics and Business Informatics \\Faculty of Economics, Technical University of Ko\v{s}ice \\ N\v{e}mcovej 32, 040 01 Ko\v{s}ice, Slovakia}
[julius.czap@tuke.sk]

\newauthor{Jakub Przyby\l o\footnote{~This work was partially supported by the Polish Ministry of Science and Higher Education.}}{J. Przyby\l o}
{AGH University of Science and Technology\\ Faculty of Applied Mathematics\\ al. A. Mickiewicza 30, 30-059 Krakow, Poland}
[jakubprz@agh.edu.pl]

\newauthor{Erika \v Skrabu\v l\'akov\'a\footnote{~This work was supported by the Slovak Research and Development Agency under the contract No. APVV-0482-11, by the grants VEGA 1/0529/15 and KEGA 040TUKE4/2014.}}{E. \v Skrabu\v l\'akov\'a}
{Institute of Control and Informatization of Production Processes\\ Faculty of Mining, Ecology, Process Control and Geotechnology\\ Technical University of Ko\v sice, Ko\v sice, Slovakia}
[erika.skrabulakova@tuke.sk]

\title{On an extremal problem in the class of 1-planar graphs}


\classnbr{05C10, 05C42, 05C62}

\begin{document}

\begin{abstract}
A graph $G=(V,E)$ is called 1-planar if it admits a drawing in the plane such that each edge is crossed at most once. In this paper, we study bipartite $1$-planar graphs with prescribed numbers of vertices in partite sets. Bipartite 1-planar graphs are known to have at most $3n-8$ edges, where $n$ denotes the order of a graph. We show that maximal-size bipartite $1$-planar graphs which are almost balanced have not significantly fewer edges than indicated by this upper bound, while the same is not true for unbalanced ones.
We prove that maximal possible sizes of bipartite $1$-planar graphs whose one partite set is much smaller than the other one tends towards $2n$ rather than $3n$. In particular, we prove that if the size of the smaller partite set is sublinear in $n$, then $|E|=(2+o(1))n$, while the same is not true otherwise.
\end{abstract}

\section{Introduction}
One of the general questions in extremal graph theory can be formulated in the following way: Given a family $\mathcal G$ of graphs, what is the maximum number of edges of an $n$-vertex graph $G\in\mathcal G$? One of the fundamental results in this area is the Theorem of Tur\'an, which states that if $\mathcal G$ is the family of $k$-clique-free graphs, then the maximum number of edges of an $n$-vertex graph $G \in \mathcal G$ is at most $\frac{(k-2)n^2}{2(k-1)}$. Tur\'an's theorem was rediscovered many times and has many corollaries. For $k=3$ we obtain Mantel's theorem: the maximum number of edges of an $n$-vertex bipartite graph is at most $\frac{n^2}{4}$.

By prescribing the family $\mathcal G$ we can study different classes of graphs. If $\mathcal G$ is a family of planar graphs, then from the Euler's formula we obtain that any $n$-vertex planar graph ($n\ge 3$) contains at most $3n-6$ edges. More strongly, any $n$-vertex planar graph can be extended to an $n$-vertex planar graph with $3n-6$ edges. Similar proposition holds for bipartite planar graphs: any $n$-vertex bipartite planar graph ($n\ge 3$) contains at most $2n-4$ edges, moreover, every $n$-vertex bipartite planar graph can be extended to an $n$-vertex bipartite planar graph with $2n-4$ edges.

If a graph is not planar, then each its drawing in the plane contains some crossings of its edges. If a graph $G$ can be drawn in the plane so that each of its edges is crossed by at most one other edge, then it is 1-planar. It is known \cite{fm,pt,sch} that any $n$-vertex 1-planar graph ($n\ge 3$) has at most $4n-8$ edges, but not every $n$-vertex 1-planar graph can be extended to an $n$-vertex 1-planar graph with $4n-8$ edges, see \cite{begghr}. %In \cite{ch1} it is proved that every 1-planar drawing of a 1-planar graph has at most $n-2$ crossings. Consequently, every $n$-vertex bipartite 1-planar graph has at most $3n-6$ edges (if we remove one crossed edge for each crossing, then we obtain a drawing without crossings, that is, a planar graph). To our knowledge, no $n$-vertex bipartite 1-planar graph ($n\ge 3$) with $3n-6$ edges is known.

In this paper we deal with the family of bipartite 1-planar graphs. We consider the problem of finding a bipartite 1-planar graph with given sizes of partite sets which has the largest number of edges among all such graphs. It is known \cite{k} that any $n$-vertex bipartite 1-planar graph has at most $3n-8$ edges for even $n\not =6$ and at most $3n-9$ edges for odd $n$ and for $n=6$.
%
%We include our exemplary construction 
Our exemplary construction confirming that these upper bounds are sharp are included at the end of Section~\ref{SectionAlmostBalanced}
(and in Lemma \ref{xyz}).
The maximal possible number of edges in such a graph keeps also relatively close to $3n$ when the cardinalities of its partite sets  
%A similar is still true when those sets 
are almost even, see %e.g. 
Lemma %\ref{odd} and 
\ref{xyz} below.
On the other hand we notice that as graphs investigated get more unbalanced
(i.e., one partite set becomes much smaller than the other) then this value
drops, see Corollary~\ref{c:lower}, and tends towards the double of the order.
Investigating this process more thoroughly,
due to Corollary~\ref{c:lower} and Lemma~\ref{ge} we are in fact able to precisely describe
for what proportions of the sizes of the partite sets
we may observe this phenomenon,
see comments in the concluding section.

%-----------------------
%{\color{red}
%We prove that the known general upper bound for this number is optimal up to additive constant at most two, see Lemma \ref{3n-8}, and is very close %to triple order of the graph. This is witnessed by bipartite graphs with even partite sets.
%A similar is still true when those sets are almost even, see Lemmas \ref{odd} and \ref{xyz}.
%On the other hand we notice that as graphs investigated get more unbalanced
%(i.e., one partite set becomes much smaller than the other) then this value
%drops, see Corollary~\ref{lower}, and tends towards the double of the order.
%Investigating this process more thoroughly,
%due to Corollary~\ref{lower} and Lemma~\ref{ge} we are in fact able to precisely describe
%for what proportions of the sizes of the partite sets
%we may observe this phenomenon,
%see comments in the concluding section.
%}
%-----------------------

Our results also partially answer the question of \'E. Sopena \cite{s}: How many edges we have to remove from the complete bipartite graph with given sizes of the partite sets to obtain a 1-planar graph? Observe that, this question is equivalent to our problem.

%Note that 1-planar graphs have many applications, for example, visualization of UML diagrams or displaying protein structure in bioinformatics (in these cases it is important to find a drawing which is from a geometric point of view as simple as possible).

\section{Notation}
In this paper we consider simple graphs. We use the standard graph theory terminology by \cite{diestel}. We use $V(G)$ and $E(G)$ to denote the vertex set and the edge set of a graph $G$, respectively.
The \emph{degree} of a vertex $v$ is denoted by $\deg(v)$. A vertex of degree $k$ is called a \emph{$k$-vertex}. Similarly, a face (of a plane graph) of size $k$  is called a \emph{$k$-face}.

We will use the following notation introduced in \cite{fm}. Let $G$ be a 1-planar graph and let $D=D(G)$ be a 1-planar drawing of $G$ (that is, a drawing of $G$ in the plane in which every edge is crossed at most once; we will also assume that no edge is self-crossing and adjacent edges do not cross). Given two non-adjacent edges $pq$,  $rs$ $\in E(G)$, the \textit{crossing} of $pq$, $rs$ is the common point of two arcs $\stackrel{\frown}{pq}$, $\stackrel{\frown}{rs}$ $\in D$ (corresponding to edges $pq$, $rs$). Denote by $C=C(D)$ the set of all crossings in $D$ and by $E_0$ the set of all non-crossed edges in $D$. The \textit{associated plane graph} $D^{\times}=D^{\times}(G)$ of $D$ is the plane graph such that $V(D^{\times})=V(D)\cup C$ and $E(D^{\times})=E_0 \cup \{xz,yz|xy \in E(D)-E_0, z\in C,z\in xy\}$. Thus, in $D^{\times}$, the crossings of $D$ become new vertices of degree 4; we call these vertices \emph{false}. Vertices of $D^{\times}$ which are also vertices of $D$ are called \emph{true}. Similarly, the edges and faces of $D^{\times}$ are called false, if they are incident with a false vertex, and true otherwise.

Note that a 1-planar graph may have different 1-planar drawings, which lead to non-isomorphic associated plane graphs.

\section{Unbalanced bipartite 1-planar graphs}

Let $G$ be a bipartite 1-planar graph such that the partite sets of $G$ have sizes $x$ and $y$. In this part of the paper we show that if $x$ is small compared to $y$, then the maximal number of edges in a corresponding bipartite 1-planar graph $G$ shall tend towards $2|V(G)|$ rather than staying close to $3|V(G)|$. %This phenomenon shall be exhibited and investigated more deeply in the following observations.

\subsection{An upper bound for the number of edges}
The following assertion improves the result of \cite{ch1} (stating that any 1-planar drawing of an $n$-vertex 1-planar graph has at most $n-2$ crossings) when $x$ is small compared to $y$.

\begin{lemma}\label{6x-12}
Let $G$ be a bipartite 1-planar graph such that the partite sets of $G$ have sizes $x$ and $y$, $2\le x\le y$. Then $G$ has a $1$-planar drawing with at most $6x-12$ crossings.
\end{lemma}

\begin{proof}
Color the vertices of $G$ in the smaller partite set with black and the rest of the vertices with white. Among all possible 1-planar drawings of  $G$, we denote by $D$ a drawing that has the minimum number of crossings and by $D^\times$ its associated plane graph. Color the false vertices with red.

Now we extend $D^\times$ in the following way. Let $v$ be a false vertex incident with black vertices $v_1$ and $v_2$. We draw a new edge $v_1v_2$ without introducing any crossings by following the  edges $v_1v$ and $v_2v$ from $v_1$ and $v_2$ until they meet in a close neighborhood of $v$, see Figure \ref{cross} for illustration.

\begin{figure}[ht!]
\centerline{
\begin{tabular}{c}
\includegraphics{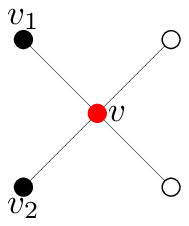}
\end{tabular}
$\rightarrow$
\begin{tabular}{c}
\includegraphics{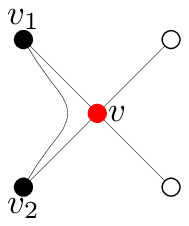}
\end{tabular}
}
\caption{The extension of $D^\times$.}
\label{cross}
\end{figure}

For every false vertex $v$ we draw a new edge $v_1v_2$ as described. Note that the new drawing might contain parallel edges. Denote the new (multi)graph by $H$.

Let $H'$ be a subgraph of $H$ (with the same embedding) induced by the black and red vertices. First we show that if the (multi)graph $H'$ has a separating 2-cycle (i.e. a cycle whose interior and  exterior contain a vertex), then its interior and also exterior contain at least one black vertex each.
To see that, assume, w.l.o.g., that $H'$ contains a separating 2-cycle which has only red vertices in the interior. This 2-cycle is a separating cycle also in $H$ since it consists of two edges which join black vertices.  The red vertices correspond to crossings, therefore there are some white vertices in the interior of this separating 2-cycle in $H$. No black vertex is in the interior of this cycle, hence all edges which join white vertices from this interior with black vertices could have been drawn without edge crossings, a contradiction with the minimality of the number of crossings in the considered 1-planar drawing $D$.

If for every 2-cycle of $H'$ with an empty interior or exterior we remove one edge incident with it, then we obtain the graph $H''$. We say that an edge of $H''$ is black if both its endvertices are black. Observe that every red vertex is incident with a 3-face in $H''$ (similarly as in $H'$). Moreover, every such 3-face is incident with one black edge. Therefore, the number of red vertices in $H''$ is at most the double of black edges in $H''$.

Now consider the subgraph  of $H''$ induced by the black vertices. This (multi) graph can be extended to a triangulation by introducing additional edges (without inserting new vertices) because even if it contains a 2-cycle, then its interior and exterior contain a vertex. This triangulation has at most $3x-6$ edges (because it has $x$ vertices). From this it follows that the graph $H''$ has at most $2(3x-6)$ red vertices. Consequently, the number of crossings in $D$ is at most $6x-12$.
\end{proof}

\begin{cor}\label{c:lower}
If $G$ is a bipartite 1-planar graph such that the partite sets of $G$ have sizes $x$ and $y$, $2\le x\le y$, then $|E(G)|\le 2|V(G)|+6x-16$.
\end{cor}

\begin{proof}
Lemma \ref{6x-12} implies that by removing at most $6x-12$ edges from $G$ we can get a planar graph. This planar graph is also bipartite. Thus, it has at most $2|V(G)|-4$ edges. Consequently, the number of edges of $G$ is at most $2|V(G)|+6x-16$.
\end{proof}

Note that, the bound in Corollary \ref{c:lower} is tight for $x=2$ (in this case $G$ is a  bipartite planar graph). For $x=3$ we can obtain a tight upper bound by a different approach.

\begin{lemma}\label{3y}
If $G$ is a bipartite 1-planar graph such that the partite sets of $G$ have sizes $3$ and $y\ge3$, then $|E(G)|\le 2|V(G)|$. Moreover, this bound is tight.
\end{lemma}

\begin{proof}
Let $V_1$ and $V_2$ be the partite sets of $G$, where $|V_1|=3$. In \cite{ch2} it is proved that the complete bipartite graph $K_{3,7}$ is not 1-planar. Therefore, there are at most six vertices of degree three in $V_2$ (and the remaining vertices have degree at most two). Consequently, $|E(G)|= \sum_{v \in V_2}\deg(v)\le 6 \cdot 3+ (|V(G)|-9)\cdot 2=2|V(G)|$.

$K_{3,6}$ is the smallest bipartite 1-planar graph for which the upper bound is attained. By adding 2-vertices to the larger partite set we can obtain more such graphs, see Figure \ref{K36+}.

\begin{figure}[ht]
\begin{center}
\includegraphics[angle=90]{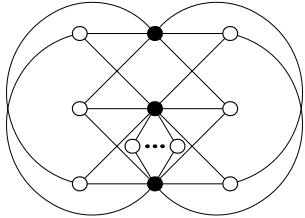}
\caption{A bipartite 1-planar graph $G$ with $3+y$ vertices and $2|V(G)|$ edges.}
\label{K36+}
\end{center}
\end{figure}
\end{proof}

\subsection{Lower bound for the number of edges}
\begin{lemma}\label{ge}
Let $x,y$ be integers such that $x\ge3$  and $y\ge6x-12$. Then there exists a bipartite 1-planar graph $G$ such that the partite sets of $G$ have sizes $x$ and $y$ and $|E(G)|\ge2|V(G)|+4x-12$.
\end{lemma}

\begin{proof}
First assume that $y=6x-12$.

Let $T$ be a triangulation on $x$ vertices. From the Euler's formula it follows that every triangulation on $x$ vertices has $2x-4$ faces. Let $T'$ be a graph obtained from $T$ by inserting a configuration $W_3$ depicted in Figure \ref{insert} into each of its faces.

\begin{figure}[ht!]
\centerline{
\begin{tabular}{c}
\includegraphics{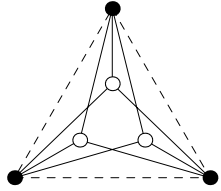}
\end{tabular}
}
\caption{The configuration $W_3$.}
\label{insert}
\end{figure}

Let $G$ be a graph obtained from $T'$ by removing the original edges of $T$. Clearly, $G$ is a bipartite 1-planar graph (the original vertices of $T$ form the first partite set and the added vertices form the second partite set). It is a routine matter to check that $|V(G)|=7x-12$ and $|E(G)|=18x-36=2|V(G)|+4x-12$.

\medskip
Now suppose that $y=6x-12+k$ for some $k\ge 1$. In this case we take the 1-planar drawing of $G$ (as it is defined above), next we add $k$ vertices to a 4-face of $G^\times$ and finally we join (without edge crossings) each of them with the two true vertices of this face. In such a way we obtain a new bipartite 1-planar graph which has $7x-12+k$ vertices and $18x-36+2k$ edges. Hence, $|E(H)|=2|V(H)|+4x-12$.
\end{proof}

%Since every planar graph is also 1-planar we immediately have the following proposition.
%
%\begin{prop}\label{bipplan}
%For any positive integers $x \ge 2$, $y\ge 2$ there is a bipartite 1-planar graph $G$ such that the partite sets of $G$ have sizes $x$, $y$ and $|E(G)|=2|V(G)|-4$.
%\end{prop}
%
%\begin{proof}
%Let $C$ be a fixed embedding of a cycle $v_1v_2v_3v_4v_1$ on four vertices. Add $x-2$ vertices to the inner part of $C$ and $y-2$ vertices to the outer part of $C$. Next join the vertices in the inner part with $v_1$ and $v_3$, thereafter join the vertices in the outer part with $v_2$ and $v_4$. Clearly, we can join the vertices so that no crossing arise. In such a way we obtain a bipartite planar graph $G$ such that its partite sets have sizes $x$ and $y$. The graph $G$ has $x+y-4$ vertices of degree two, two vertices of degree $x$ and two vertices of degree $y$. Hence, $|E(G)|=\frac12\cdot\sum_{v \in V(G)}\deg(v)=2x+2y-4=2|V(G)|-4$.
%\end{proof}

\section{Almost balanced bipartite 1-planar graphs\label{SectionAlmostBalanced}}

\begin{lemma}\label{le}
Let $x,y$ be integers such that $x\ge3$, $y\ge6$ and $x\le y\le6x-12$. Then there exists a bipartite 1-planar graph $G$ such that the partite sets of $G$ have sizes $x$, $y$ and $|E(G)|\ge\frac52 |V(G)|+\frac x2-\frac{17}{2}$.
\end{lemma}

\begin{proof}
First assume that $y=6r$ for some integer $r\ge 1$. Let $T$ be a triangulation on $\frac{6r}{6}+2=\frac y6 +2$ vertices. Color the vertices of $T$ with black. Let $x=\frac y6 +2+3s+t$, where $s\ge0$ and $t \in \{0,1,2\}$ ($x\ge \frac y6 +2$ since $y\le6x-12$).
Let $T'$ be a graph obtained from $T$ by inserting a configuration $B_3$ depicted in Figure \ref{b} into $s$ faces, a configuration $B_2$ into one face if $t=2$,  a configuration $B_1$ into one face if $t=1$ and the configuration $B_0$ to the other faces of $T$ and removing the original edges of $T$. This modification is possible if and only if $T$ has at least $s+1$ faces (or $s$ faces if $t=0$). The number of faces of $T$ is $2(\frac y6+2)-4=\frac y3$, so we need to show that $s+1 \le \frac y3$. From $x=\frac y6 +2+3s+t$ and $x\le y$ we obtain $\frac 45+\frac 65 s+\frac 25 t\le \frac y3 $. The inequality $s+1 \le \frac 45+\frac 65 s+\frac 25 t$, or equivalently $1\le s+2t$ does not hold if and only if $t=s=0$. But in this case the inequality $s+1 \le \frac y3$ trivially holds. Observe that $T'$ has $(\frac y6+2)+3s+t=x$ black vertices and $3\cdot\frac y3=y$ white vertices, moreover it has $3(x+y-(\frac y6+2))=\frac 52 (x+y)+\frac x2-6$ edges.

\begin{figure}[ht!]
\centerline{
\begin{tabular}{c}
\includegraphics{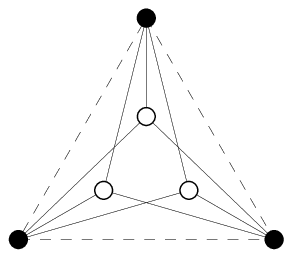}\\$B_0$
\end{tabular}
\begin{tabular}{c}
\includegraphics{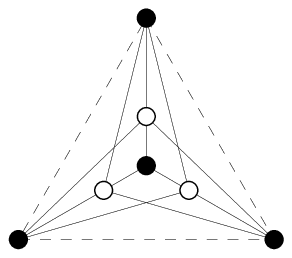}\\$B_1$
\end{tabular}
\begin{tabular}{c}
\includegraphics{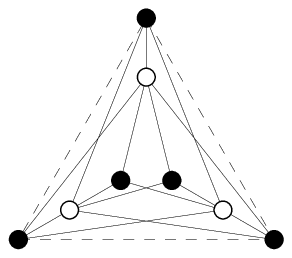}\\$B_2$
\end{tabular}
\begin{tabular}{c}
\includegraphics{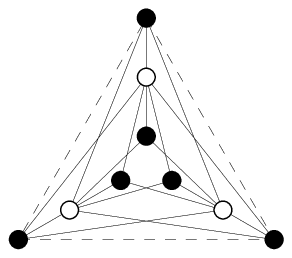}\\$B_3$
\end{tabular}
}
\caption{The configurations $B_0$, $B_1$, $B_2$ and $B_3$.}
\label{b}
\end{figure}

If $y=6r+u$, where $r\ge1$ and $u \in \{1,2,3,4,5\}$, then we proceed similarly as above. In this case $T$ is a triangulation on $\frac{6r+6}{6}+2$ vertices. Using this triangulation we obtain (by the same construction as previously) a bipartite 1-planar graph on $x+(6r+6)$ vertices and $3x+\frac52 \cdot(6r+6)-6$ edges. Note that if $x \not=11$ or $y \not=11$, then we must have inserted the configuration $B_0 $ into at least two faces of $T$ according to our construction, since otherwise the graph $T'$ has at least ($\frac{6r+6}{6}+2)+3\cdot (\frac{6r+6}{3}-2)+1=7r+4$ black vertices. At the same time, the graph $T'$ has $x\le 6r+u$ black vertices, and therefore $7r+4\le6r+u$, or equivalently $r+4 \le u$. This inequality in turn has only one solution, namely $r=1$ and $u=5$. This implies $x=y=11$.

If we remove $6-u$ white vertices of two configurations of type $B_0$, then we obtain a bipartite 1-planar graph with $x+y$ vertices and $3x+\frac52 \cdot(6r+6)-6-3(6-u)=\frac52\cdot(x+6r+u)+\frac x2-9+\frac u2\ge\frac52\cdot(x+y)+\frac x2-\frac{17}{2}$ edges.

If $x=y=11$, then there is a bipartite 1-planar graph $G$ such that the partite sets of $G$ have sizes $x$ and  $|E(G)|=3|V(G)|-8$, see \cite{k}.
\end{proof}

%For the sake of improvement
The following result provides a lower-bound improvement
in the case when $G$ is very close to being balanced.
%we prove the following.
%In the case when $G$ is more balanced we prove the following.

\begin{lemma}\label{xyz}
Let $x,y,z$ be positive integers such that $x\ge3$, $y=x+z$, $z\ge0$. Then there exists a bipartite 1-planar graph $G$ such that the partite sets of $G$ have sizes $x$ and $y$ and $|E(G)|=3|V(G)|-8-z$.
\end{lemma}

\begin{proof}
First we take a 1-planar drawing of a bipartite 1-planar graph $G$ on $x+x$ vertices and $6x-8$ edges, see e.g. \cite{k}. The edges of this drawing divide the plane into some regions. We insert $z$ vertices to the region which is incident with two vertices from the same partite set and join the added vertices with these two ones (without edge crossings), see Figure \ref{3n-}.

\begin{figure}[ht!]
\centerline{
\begin{tabular}{c}
\includegraphics{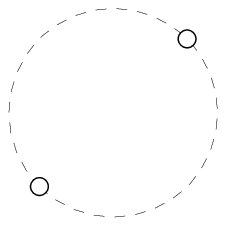}
\end{tabular}
$\rightarrow$
\begin{tabular}{c}
\includegraphics{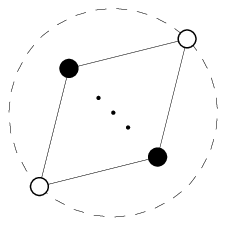}
\end{tabular}
}
\caption{The extension of $G$.}
\label{3n-}
\end{figure}

Let $H$ denote the obtained graph. Clearly, $H$ is a bipartite 1-planar graph with $|V(H)|=|V(G)|+z=2x+z$ and $|E(H)|=|E(G)|+2z=6x-8+2z=3|V(H)|-8-z$.
%\medskip
%If $x$ is odd, then first we construct a bipartite 1-planar graph $G$ on $(x-1)$ white and $(x-1)$ black vertices and $6x-14$ edges (as above). Thereafter we insert $z+1$ black vertices to the inner part of the cycle $C_1$ and join all of them with the two white vertices (without edge crossings). Finally, we add one white vertex to the inner part of $C_1$ and join it with four black vertices such that we cross each edge at most once, see Figure \ref{final}.
%\begin{figure}[ht!]
%\centerline{
%\begin{tabular}{c}
%\includegraphics{final.eps}
%\end{tabular}
%}
%\caption{The last step of the construction.}
%\label{final}
%\end{figure}
%
%The obtained graph has $2x+z$ vertices and $6x-8+2z$ edges.
\end{proof}

For the sake of completeness we describe a construction of a bipartite 1-planar graph $G$ on $x+x$ vertices and $3(x+x)-8$ edges.

Let $x=2k$ for some positive integer $k$. If $k=1$, then $G$ is a cycle on four vertices. Let $H$ be a graph consisting of $k\ge2$ cycles $C_i=x_{1,i}y_{1,i}x_{2,i}y_{2,i}x_{1,i}$ on four vertices, $i=1,\dots,k$. Take an embedding of $H$ such that the cycle $C_i$ is in the inner part of $C_j$
(i.e., inside the bounded part of the plane with boundaries determined by $C_j$) if $i<j$.  Next we extend this drawing of $H$ by adding the edges $x_{1,i}y_{1,i+1}$, $x_{1,i}y_{2,i+1}$,  $x_{2,i}y_{1,i+1}$, $x_{2,i}y_{2,i+1}$, $x_{1,i+1}y_{1,i}$, $x_{1,i+1}y_{2,i}$, $x_{2,i+1}y_{1,i}$ and $x_{2,i+1}y_{2,i}$ for $i=1,\dots,k-1$ so that the edge $x_{\ell,i+1}y_{j,i}$ crosses the edge $x_{\ell,i}y_{j,i+1}$ for $j,\ell\in\{1,2\}$, $i=1,\dots, k-1$, see Figure \ref{nn} for illustration.

\begin{figure}[ht!]
\centerline{
\begin{tabular}{c}
\includegraphics{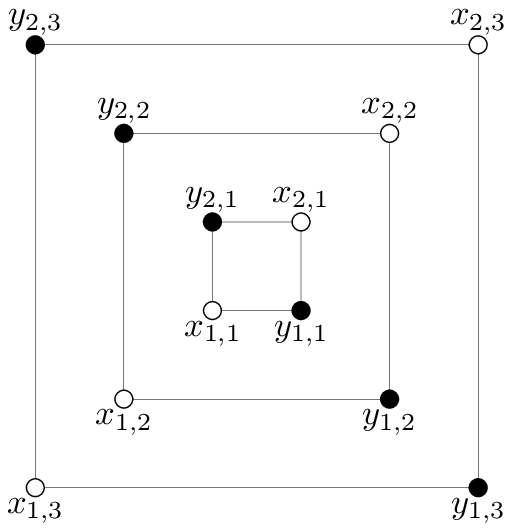}
\end{tabular}
$\rightarrow$
\begin{tabular}{c}
\includegraphics{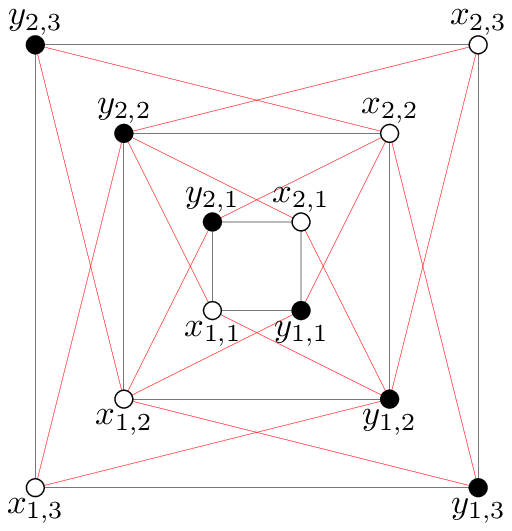}
\end{tabular}
}
\caption{A construction of a bipartite 1-planar graphs with $x+x$ vertices and $6x-8$ edges for $x$ even.}
\label{nn}
\end{figure}

The new graph has $4k$ vertices of degree six and eight vertices of degree four, therefore it has $12k-8$ edges.

If $x=2k+1$, then we modify the graph obtained for $x=2k$ in the following way. First we remove the edges $x_{1,i}y_{1,i-1}$ for $i=2,3,\dots, k$ and the edges $x_{1,i}y_{1,i}$ for  $i=2,3,\dots, k-1$. Thereafter we add the edges $x_{1,i}y_{1,i+2}$ for $i=1,2,\dots, k-2$ and the edges $x_{1,i}y_{1,i+3}$ for $i=1,2, \dots, k-3$. Finally, we add a vertex to the region which is incident with the vertices $x_{1,1},y_{1,1}, y_{1,2}$ and join it with the vertices $y_{1,1},y_{1,2},y_{1,3}, y_{2,1}$; then add a vertex to the region which is incident with the vertices $x_{1,k-1},x_{1,k}, y_{1,k}$ and join it with the vertices $x_{1,k-2},x_{1,k-1},x_{1,k}, x_{2,3}$ as it is depicted in Figure \ref{nnuj}.

We removed $2k-3$ edges and added two vertices and $2k+3$ edges. Therefore the obtained bipartite 1-planar graph has $4k+2$ vertices and $(12k-8)-(2k-3)+(2k+3)=12k-2=3(4k+2)-8$ edges.

%\begin{lemma}\label{odd}
%Let $x\ge3$ be a positive odd integer. Then there exists a bipartite 1-planar graph $G$ such that the partite sets of $G$ have sizes $x$ and $|E(G)|=3|V(G)|-9$.
%\end{lemma}

%
%\begin{proof}
%First we use the construction defined in the proof of Lemma \ref{3n-8} to obtain a bipartite 1-planar graph $G$ on $(x-1)+(x-1)$ vertices and $6x-14$ edges. Next we insert two vertices to the inner part of the cycle $C_1$ and add five edges as it is depicted in Figure \ref{f:odd}.
%
%\begin{figure}[ht!]
%\centerline{
%\begin{tabular}{c}
%\includegraphics{c1+2v.eps}
%\end{tabular}
%}
%\caption{The cycle $C_1$ with two new vertices and five new edges.}
%\label{f:odd}
%\end{figure}
%
%The obtained graph has $2x$ vertices and $6x-9$ edges.
%\end{proof}

%We believe that the bound from the construction (in the proof of Lemma \ref{3n-8}) is tight, cf. Theorem~\ref{bip1plan}.
%\begin{con}
%If $G$ is a bipartite 1-planar graph on at least four vertices, then $|E(G)|\le 3|V(G)|-8$.
%\end{con}

\begin{figure}[ht!]
\centerline{
\begin{tabular}{lll}
\includegraphics{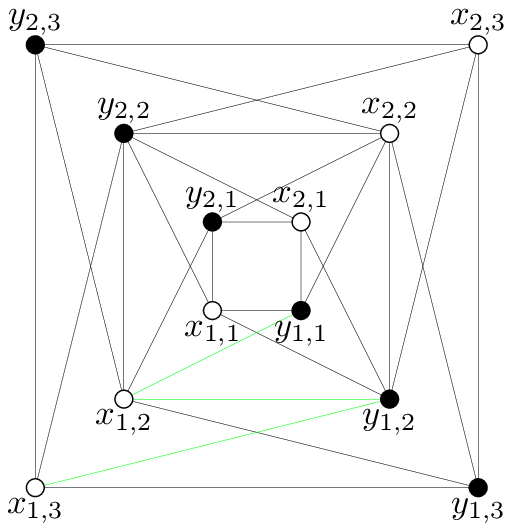}
&&
\includegraphics{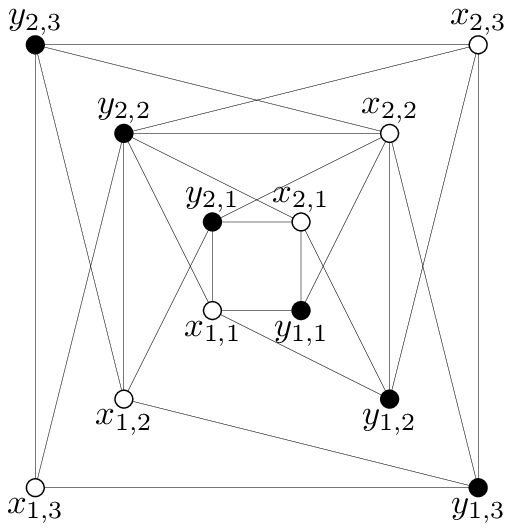}
\\
\includegraphics{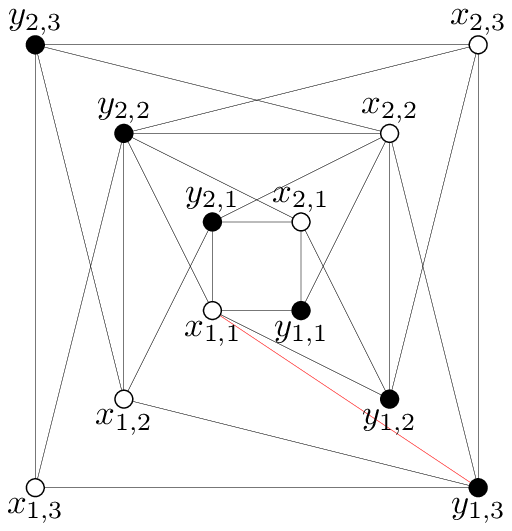}
&&
\includegraphics{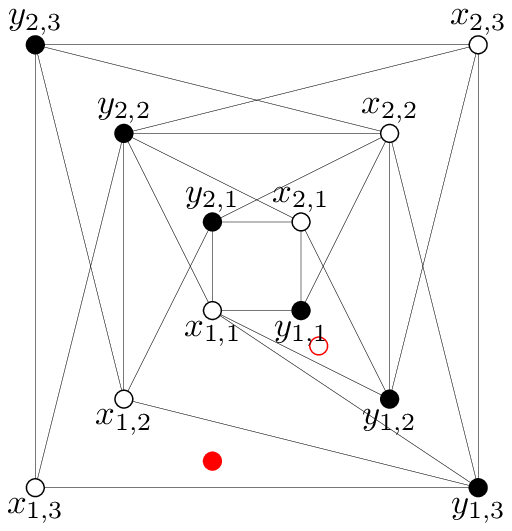}
\\
\includegraphics{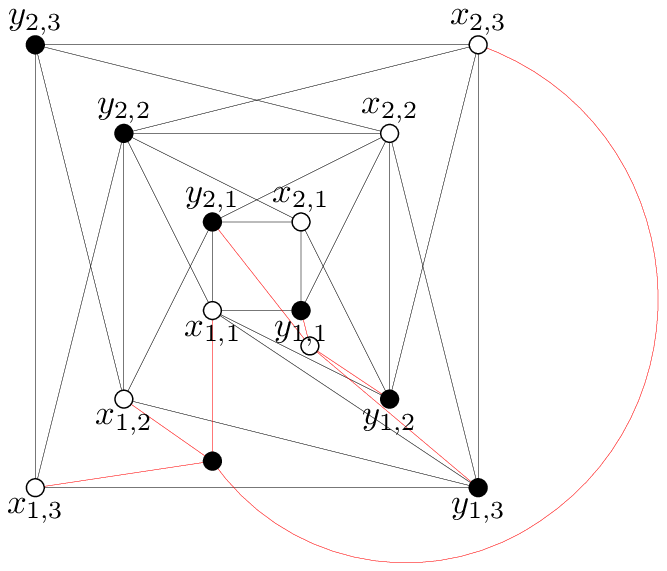}
\end{tabular}
}
\caption{A construction of a bipartite 1-planar graphs with $x+x$ vertices and $6x-8$ edges for $x$ odd.}
\label{nnuj}
\end{figure}

\section{Comments}

For given integers $x,y$, $x\leq y$, let $G_{x,y}$ be a bipartite 1-planar graph with partite sets of sizes $x$ and $y$ with the maximal number of edges. Denote by $g_{x,y}$ the size of this graph.
By \cite{k}, we always have $g_{x,y}\leq 3|V(G_{x,y})|-8$. It follows from Lemma~\ref{xyz}, that $g_{x,y}$ keeps close to $3|V(G_{x,y})|$ if $G_{x,y}$ is balanced enough, i.e., when $x$ is not significantly smaller than $y$.
The larger is the difference between $x$ and $y$, the smaller multiplicity of $|V(G_{x,y})|$ expresses $g_{x,y}$, as exemplified by Corollary \ref{c:lower}. %... and especially ....
By Lemma~\ref{ge} however it never drops under $2|V(G_{x,y})|$ if $x\geq 3$. This implies a natural question on how the ratio $g_{x,y}/|V(G_{x,y})|$ depends on the proportion of $x$ and $y$, in particular, when this ratio gets closer to $2$ rather than $3$.

Our research was thus motivated by the wish
to reveal a kind of threshold
for $x$, given by a function of $y$ under which $g_{x,y}/|V(G_{x,y})|$ actually
converges to $2$ as $y$ tends to infinity.
The results of this paper imply the following solution of this problem.
Suppose $x=f(y)$ is any fixed linear function of $y$ (e.g., $x=0.1y$),
then by Corollary~\ref{c:lower} and Lemma~\ref{ge} there exist constants $c_1$ and $c_2$ such that
$$(2+c_1)|V(G_{x,y})|\leq g_{x,y}\leq (2+c_2)|V(G_{x,y})|$$
(for $y$ large enough). If on the other hand, $x$ is expressed by any sublinear function of $y$,
then $g_{x,y}/|V(G_{x,y})|=2+o(1)$, cf. Corollary~\ref{c:lower}.

Note also that if $x\geq \frac{1}{6}y+2$, then by Lemma~\ref{le}, $g_{x,y}$ exceeds $\frac{5}{2}|V(G_{x,y})|$.
On the other hand, we believe that for $x\leq \frac{1}{6}y+2$, our construction from Lemma~\ref{ge} is optimal and thus conclude by posing the following conjecture.

\begin{con}
For any integers $x,y$ such that $x\ge 3$  and $y\ge6x-12$, every bipartite 1-planar graph $G$ with partite sets of sizes $x$ and $y$ has at most $2|V(G)|+4x-12$ edges.
\end{con}

%\section{Conclusion}
%
%In this paper we proved the following propositions.
%\begin{theorem}
%Let $x,y$ be integers such that $x\ge3$ and $y\ge6$. There exists a bipartite 1-planar graph $G$ such that %the partite sets of $G$ have sizes $x$ and $y$ and
%
%$$|E(G)|\ge\left\{\begin{array}{rl}  2|V(G)|+4x-12 &\textrm{if\quad }y\ge6x-12,\\
%\\
%                                     \frac52 |V(G)|+\frac x2-\frac{17}{2} &\textrm{if\quad }y\le6x-12.
%                 \end{array}\right.$$
%\end{theorem}
%
%\begin{proof}
%It follows from lemmas \ref{ge} and \ref{le}.
%\end{proof}
%
%\begin{theorem}
%Let $x,y,z$ be positive integers such that $x\ge3$, $y=x+z$, $z\ge0$. Then there exists a bipartite %1-planar graph $G$ such that the partite sets of $G$ have sizes $x$ and $y$ and $|E(G)|\ge3|V(G)|-9-z$.
%\end{theorem}
%
%\begin{proof}
%It follows from lemmas \ref{3n-8}, \ref{odd} and \ref{xyz}.
%\end{proof}
%
%\begin{theorem}
%If $G$ is a bipartite 1-planar graph such that the partite sets of $G$ have sizes $x$ and $y$, $x\le y$, %then $|E(G)|\le \min\{2|V(G)|+6x-16;\, 3|V(G)|-6\}$.
%\end{theorem}
%
%\begin{proof}
%It follows from Theorem \ref{bip1plan} and Corollary \ref{lower}.
%\end{proof}

\end{document}